\documentclass{amsart}
\usepackage{amssymb,amsmath,bbold}
\usepackage{srcltx}

\newtheorem{theorem}{Th\'eor\`eme}[section]
\newtheorem{lemma}[theorem]{Lemma}
\newtheorem{proposition}[theorem]{Proposition}

\theoremstyle{definition}
\newtheorem{definition}[theorem]{Definition}

\theoremstyle{remark}

\theoremstyle{conjecture}

\newcommand \F{\mathcal{F}}

\numberwithin{equation}{section}

\begin{document}

\title{Bernstein processes described by forward and backward heat equations} 

\author{{\bf Mohamad Houda \rm}}

\date{01/08/2013}

\setcounter{section}{0}

\begin{abstract}
We present some results on Bernstein processes which are Brownian diffusions that appear in Euclidean Quantum Mechanics: We express the distributions of these processes with the help of those of 
Bessel processes. We then determine two solutions of the dual equation of the heat equation with potential. This continues the calculation of Lescot for the solution with an Ornstein-Uhlenbek process \cite{Lescot2013}.

\vspace{2mm}
Keywords: Bessel processes, Bernstein processes, interest rate models.

\bigskip

\end{abstract}

\maketitle
\markright{Bernstein processes described by forward and backward heat equations} 
\setcounter{section}{0}
\numberwithin{equation}{section}

\section{Introduction}

\textit{Bernstein processes} or reciprocal diffusions combine two dynamic behaviors : forward and backward. These processes are useful tools of stochastic quantum mechanics, and also find applications in various other fields.

We give in this paper an explicit expression for the distribution density of a special Bernstein process. This process is similar to the Cox-Ingersoll-Ross process of financial mathematics. Indeed, more recently 
it has appeared that each \textit{one-factor affine interest rate model} (in the sense of Leblanc-Scaillet) could be described using such a Bernstein process.

The groundbreaking idea of replacing the complex Schr\"odinger equation by forward and backward heat equations in duality goes back to Schr\"odinger.

Professor Paul Lescot has computed two solutions of the dual equation by using the Gaussian character of Ornstein-Uhlenbek process and Brownian motion. In this paper, this goal will be reached by using another class of diffusion which is that of \textit{Bessel processes}. 

Let us first give some definitions and recall some preliminary results.

\begin{definition}
A \emph{Bernstein process} $z$ is a process satisfying a stochastic differential equation of the form 
$$
dz(t)=\theta dw(t)+\tilde{B}(t , z(t))dt \,\, 
$$
relatively to the canonical filtration of the Brownian motion $w$ and the dual stochastic differential equation  
$$
dz(t)=\theta dw_{*}(t)+\tilde{B_{*}}(t , z(t))dt 
$$
relatively to the canonical decreasing filtration of another Brownian motion $w_{*}$ which is independent of $w$, where 
$$
\tilde{B}\mbox{:}\equiv \theta^{2}\displaystyle\frac{\displaystyle\frac{\partial \eta}{\partial q}}{\eta}\,,
$$
$$
 \tilde{B_{*}}\mbox{:}\equiv -\theta^{2}\displaystyle\frac{\displaystyle\frac{\partial \eta_{*}}{\partial 
q}}{\eta_{*}}\,
$$ 
and, for each given $t > 0$, the law of $z(t)$ is $ \eta(t,q)\eta_{*}(t,q)dq.$The function $\eta$ is assumed to be an everywhere positive solution to
$$
\theta^{2}\displaystyle\frac{\partial \eta}{\partial t}=-\displaystyle\frac{\theta^{4}}{2}
\displaystyle\frac{\partial^{2}\eta}{\partial q^{2}}+V\eta   \,\,\,\, (\mathcal C_{1}^{(V)})\,.
$$

Similarly, $\eta_{*}$ is assumed to be everywhere positive and a solution to

$$
-\theta^{2}\displaystyle\frac{\partial \eta_{*}}{\partial t}=-\displaystyle\frac{\theta^{4}}{2}
\displaystyle\frac{\partial^{2}\eta_{*}}{\partial q^{2}}+V\eta_{*}       \,\,\,\, (\mathcal C_{2}^{(V)})\,.
$$
\end{definition}
\begin{definition}\cite[$p.454$ $(2.1)$]{Dieudonne1968}
The \emph{Bessel function} $J_\lambda$ with index $\lambda \in \mathbf C$ is defined by 
$$
J_\lambda(z) = (\frac{z}{2})^\lambda\,\sum_{n=0}^{\infty}\frac{(-z^2)^n}{2^{2n}\, n!\,\Gamma(n+\lambda+1)}\,.
$$ 
This function satisfies the {Bessel} equation with parameter $\lambda \in \mathbf C$ 
\begin{eqnarray}
z^2\ddot{\omega} + z\dot{\omega} + (z^2 - \lambda^2)\,\omega = 0
\end{eqnarray}
\end{definition}
\begin{definition}
The \emph{Bessel modified function} $I_\nu$ with index $\nu \in \mathbf R$ is defined by
$$
I_\nu(z) = i^{-\nu}J_\nu(iz)\,.
$$
This function satisfies the linear differential equation of second order
\begin{eqnarray}
z^2\ddot{I_\nu}(z) + z\dot{I_\nu}(z) -(z^2 + \nu^2)I_\nu(z) = 0\,.
\end{eqnarray}
\end{definition}
Indeed, we have $I_\nu(z) = i^{-\nu}J_\nu(iz)\,,\dot{I_\nu}(z) = i^{-\nu}\,i\,\dot{J_\nu}(iz)$ et $\ddot{I_\nu}(z) = i^{-\nu}\,i^{2} \ddot{J_\nu}(iz)\,.
$
Equation $(1.1)$ gives 
$$
(iz)^2\,i^{\nu}\,i^{-2} \ddot{I_\nu}(z) + iz\,i^{\nu}\,i^{-1}\dot{I_\nu}(z) + (-z^2-\nu^2)\,i^{\nu} I_\nu(z) = 0
$$
$$
z^2\ddot{I_\nu}(z) + z\dot{I_\nu}(z) -(z^2 + \nu^2)I_\nu(z) = 0\,,
$$
which yields $(1.2)$.

We refer to $[1]$ for details on Bessel functions.
\begin{definition}
For all  $\delta \geq 0$ and $x_0 \geq 0$, the unique solution of 

$$
Y_t = x_0 + \delta t + 2\int_0^t \sqrt{|Y_s|}dw_s.
$$ 
starting from ${x_0}$, is called \emph{squared Bessel process of dimension $\delta$}. This process will be denoted by BESQ$_{x_0}^\delta$.
\end{definition}
We refer to $[2]$ and $[5]$ for a survery on Bessel processes.
\begin{proposition} \cite[$p.441$ $(1.4)$]{RevuzYor1993} If $\bold{x_0 = 0}$, then the density $q_t^\delta$ of the law of $Y_t$ is given by

\begin{eqnarray}
q_t^\delta(0,y) = (2t)^{-\frac{\delta}{2}} \Gamma(\frac{\delta}{2})^{-1} y^{\frac{\delta}{2}-1} \exp(-\frac{y}{2t})\,,\,t > 0\,.
\end{eqnarray}
If $\bold{x_0 > 0}$, then the density $q_t^\delta$ of the law of $Y_t$ is given by
\begin{eqnarray}
q_t^\delta(x_0,y) = \frac{1}{2t}(\frac{y}{x_0})^{\frac{\nu}{2}}\,\exp(-\frac{x_0+y}{2t})\,I_\nu(\frac{\sqrt{x_0y}}{t})\,, \,t>0 
\end{eqnarray}
where $\nu \mbox{:}= \frac{\delta}{2}-1$.
\end{proposition}
\section{ON AFFINE INTEREST RATE MODELS}
Let $(\Omega, \F, (\F_t)_{t\geq 0}, Q)$ be a filtered probability space, and let $w(t)_{t\geq 0}$ be an $(\F_t)_{t\geq 0} $ standard Brownian motion.

A {\em one-factor affine interest rate model} is characterized by the instantaneous 
rate $r(t)$, satisfying the stochastic differential equation
$$ d r(t) = \sqrt{\alpha r(t) + \beta} \,
 d w(t) +(\phi - \lambda r (t))\,dt 
$$
under the risk-neutral probability $Q$ $($\cite{LeblancScaillet1998}, $p.351)$.

Assuming $\alpha \neq 0$, let us set, with the notations of Lescot \cite{Lescot2013}, 
$$
\tilde{\phi} = \phi + \frac{\lambda\beta}{\alpha}\,,
$$

$$
\delta = \frac{4 \tilde{\phi}}{\alpha}\,,
$$

$$
A = \frac{\alpha^4}{128}(\delta-1)(\delta-3)\,,
$$
and
$$
B = \frac{\lambda^2}{8}.$$ 
\begin{proposition}$($Lescot \cite{Lescot2013}$)$. Let $X_t$ be defined by 
$$ X_t=\alpha r(t)+\beta\,.
$$ 
The process $X$ is called the {\em Cox-Ingersoll-Ross process}. Define $Z(t)$ by
$$
Z(t) = \sqrt{\alpha r(t)+\beta}\,.
$$
Then $Z$ is a Bernstein process with
$$
\theta = \frac{\alpha}{2}
$$
and the potential
$$
V(t,q) = \frac{A}{q^2} + Bq^2 \,.$$
\end{proposition}
\begin{lemma}\cite[$p.314$ $(2))$]{JaeschkeYor2003}. Let $\bold{X_0 = x_0} \,,$ then
$$
X_t = e^{-\lambda t}Y(s)\,,
$$
where $s=\frac{\alpha^2(e^{\lambda t}-1)}{4\lambda}$ and $Y$ is BESQ$_{x_0}^\delta$.
\end{lemma}

Our main result is the following : 

We determine two solutions of the dual equation $(\mathcal C_{2}^{(V)})$, when the potential $V$ is written as 
$$
V(t,q) = \frac{A}{q^2} + Bq^2\,.
$$
Firstly we search $\rho_t(q)$ the density of the law of $Z_t$. Here, two cases arise : $\bold{x_0 = 0}$ or $\bold{x_0 > 0}$.
\begin{proposition}
If $\bold{x_0 = 0}$, then the density $\rho_t(q)$ of the law of $Z_t$ is written as
\begin{eqnarray*}
\rho_t(q) = \alpha^{-\delta} 2^{\frac{\delta}{2}+1} \lambda^{\frac{\delta}{2}}(e^{\lambda t}-1)^{-\frac{\delta}{2}}\Gamma(\frac{\delta}{2})^{-1} q^{\delta-1} \exp(\frac{\lambda \delta}{2} t)\exp(\frac{-2\lambda e^{\lambda t}q^2}{\alpha^2(e^{\lambda t}-1)})\displaystyle\mathbb{1}_{\{q > 0\}}\,.
\end{eqnarray*}
\end{proposition}

\begin{proof}
For fixed $t$ and for all $g$ bounded continuous,
\begin{eqnarray*}
E\left(g(X_t)\right) &=& E\left[g\left(\exp(-\lambda t) Y(s) \right) \right]\\ 
&\underset{(1.3)}=& \int_{0}^\infty g(\exp(-\lambda t)y) (2s)^{-\frac{\delta}{2}} \Gamma(\frac{\delta}{2})^{-1} y^{\frac{\delta}{2}-1} \exp(-\frac{y}{2s}) dy\\
&=& (2s)^{-\frac{\delta}{2}} \Gamma(\frac{\delta}{2})^{-1} \int_{0}^\infty g(\exp(-\lambda t) y) y^{\frac{\delta}{2}-1} \exp(-\frac{y}{2s}) dy\,.
\end{eqnarray*}
By the change of variable $x=\exp(-\lambda t)y$, we get
\begin{eqnarray*}
E\left(g(X_t)\right) &=&(2s)^{-\frac{\delta}{2}}\Gamma(\frac{\delta}{2})^{-1} \int_{0}^\infty g(x) [x \exp(\lambda t)]^{\frac{\delta}{2}-1} \exp(-\frac{e^{\lambda t}x}{2s}) \exp(\lambda t) dx\\
&=& \int_{0}^\infty g(x) (2s)^{-\frac{\delta}{2}}\Gamma(\frac{\delta}{2})^{-1} x^{{\frac{\delta}{2}-1}} \exp(\frac{\lambda \delta}{2} t)\exp(-\frac{e^{\lambda t}x}{2s}) dx.
\end{eqnarray*}
Then the density $\ell_{X_t}$ of the law of $X_t$ is given by
$$
\ell_{X_t}(x) = \alpha^{-\delta} (2\lambda)^{\frac{\delta}{2}}(e^{\lambda t}-1)^{-\frac{\delta}{2}}\Gamma(\frac{\delta}{2})^{-1} x^{{\frac{\delta}{2}-1}} \exp(\frac{\lambda \delta}{2} t)\exp(\frac{-2\lambda e^{\lambda t}x}{\alpha^2(e^{\lambda t}-1)}) \mathbb{1}_{\{x > 0\}}\,.
$$
For fixed $t$ and for all bounded continuous $\varphi$, we have
\begin{eqnarray*}
E[\varphi(Z_t)] &=& E[\varphi(\sqrt{X_t})]\\
&=& \int_0^\infty \varphi(\sqrt{x})\ell_{X_t}(x) dx\,.
\end{eqnarray*}
By the change of variable $\,q=\sqrt{x}\,,$ we get
\begin{eqnarray*}
E[\varphi(Z_t)] &=& \int_0^\infty \varphi(q) \alpha^{-\delta} (2\lambda)^{\frac{\delta}{2}}(e^{\lambda t}-1)^{-\frac{\delta}{2}}\Gamma(\frac{\delta}{2})^{-1} q^{\delta-2} \exp(\frac{\lambda \delta}{2} t)\exp(\frac{-2\lambda e^{\lambda t}q^2}{\alpha^2(e^{\lambda t}-1)})\,2 q dq\\
&=&\int_0^\infty \varphi(q) \alpha^{-\delta} 2^{\frac{\delta}{2}+1} \lambda^{\frac{\delta}{2}}(e^{\lambda t}-1)^{-\frac{\delta}{2}}\Gamma(\frac{\delta}{2})^{-1} q^{\delta-1} \exp(\frac{\lambda \delta}{2} t)\exp(\frac{-2\lambda e^{\lambda t}q^2}{\alpha^2(e^{\lambda t}-1)})\,dq\,.
\end{eqnarray*}
This proves the result.
\end{proof}
\bigskip
Let $\eta(t,q)$ be defined by
$$
\eta(t,q) = \exp(\frac{\lambda\delta t}{4}-\frac{\lambda q^2}{\alpha^2})\,q^{\frac{\delta-1}{2}}.
$$
The function $\eta$ is a solution of 
\begin{eqnarray*} 
\theta^2\frac{\partial \eta}{\partial t} = -\frac{\theta^4}{2}\frac{\partial^2 \eta}{\partial q^2} + V\eta \,\,\,\, (\mathcal C_{1}^{(V)})\,,
\end{eqnarray*}
with
$$
V(t,q) = \frac{A}{q^2} + Bq^2\,.
$$
\begin{proposition}
The function\ $\eta_*(t,q)$ defined by 
\begin{eqnarray*}
\eta_*(t,q) &=& \frac{\rho_t(q)}{\eta(t,q)}\\ 
&=& \frac{2^{\frac{\delta}{2}+1}}{\alpha^\delta} \frac{\lambda^{\frac{\delta}{2}}}{\Gamma(\frac{\delta}{2})} (e^{\lambda t}-1)^{-\frac{\delta}{2}} q^{\frac{\delta-1}{2}} \exp(\displaystyle \frac{\lambda\delta t}{4} -\frac{\lambda q^2}{\alpha^2\tanh(\frac{\lambda t}{2})})\,,
\end{eqnarray*}
satisfies the dual equation
\begin{eqnarray*} 
-\theta^2\frac{\partial \eta_*}{\partial t} = -\frac{\theta^4}{2}\frac{\partial^2 \eta_*}{\partial q^2} + V\eta_* \,\,\,\, (\mathcal C_{2}^{(V)})\,.
\end{eqnarray*}
\end{proposition}
\bigskip
\begin{proof}
Let
$$
C =\alpha^{-\delta}\, 2^{\frac{\delta}{2}+1} \lambda^{\frac{\delta}{2}}\,\Gamma(\frac{\delta}{2})^{-1}$$
and 
$$
C_1 = C (e^{\lambda t}-1)^{-\frac{\delta}{2}}\,.
$$
We have
\begin{eqnarray*}
\frac{\partial\eta_*}{\partial t} &=& \Bigg(-\frac{\delta}{2}\,(\lambda e^{\lambda t})+
\frac{\lambda\delta}{4} - \frac{\lambda q^2}{\alpha^2}(-\frac{2\lambda e^{\lambda t}}{(e^{\lambda t}-1)^2})
\Bigg)
C(e^{\lambda t}-1)^{-\frac{\delta}{2}-1}\,q^{\frac{\delta-1}{2}}\exp(\displaystyle \frac{\lambda\delta t}{4} -\frac{\lambda q^2}{\alpha^2\tanh(\frac{\lambda t}{2})})\,,
\end{eqnarray*}
\begin{eqnarray*}
\frac{\partial\eta_*}{\partial q} = \left((\frac{\delta-1}{2})\,q^{\frac{\delta-3}{2}} + q^{\frac{\delta-1}{2}}(-\frac{2\lambda q}{\alpha^2\tanh(\frac{\lambda t}{2})}) \right)C_1\,\exp(\displaystyle \frac{\lambda\delta t}{4} -\frac{\lambda q^2}{\alpha^2\tanh(\frac{\lambda t}{2})})\,,
\end{eqnarray*}
\begin{eqnarray*}
\frac{\partial^2\eta_*}{\partial q^2}&=&\Bigg(
(\frac{\delta-1}{2})(\frac{\delta-3}{2})\,q^{\frac{\delta-5}{2}}
+(\frac{\delta-1}{2})\,q^{\frac{\delta-3}{2}}(-\frac{2\lambda q}{\alpha^2\tanh(\frac{\lambda t}{2})})
+ q^{\frac{\delta-1}{2}}\,(-\frac{2\lambda}{\alpha^2\tanh(\frac{\lambda t}{2})})\\ 
&+&(\frac{\delta-1}{2})\,q^{\frac{\delta-3}{2}}(-\frac{2\lambda q}{\alpha^2\tanh(\frac{\lambda t}{2})})+ q^{\frac{\delta-1}{2}}\,(-\frac{2\lambda q}{\alpha^2\tanh(\frac{\lambda t}{2})})^2
\Bigg)C_1\,
\exp(\displaystyle \frac{\lambda\delta t}{4} -\frac{\lambda q^2}{\alpha^2\tanh(\frac{\lambda t}{2})})\,,
\end{eqnarray*}
and
\begin{eqnarray*}
V(t,q)\,\eta_*(t,q) &=&(\frac{A}{q^2} + Bq^2)\,\eta_*(t,q)\\
&=&(\frac{\alpha^4}{2^7}(\delta-1)(\delta-3)\,q^{-2} + \frac{\lambda^2}{2^3}q^2)\,\eta_*(t,q)\\
&=&\left(\alpha^{-\delta + 4}\,2^{\frac{\delta}{2}-6}\,\lambda^{\frac{\delta}{2}}(\delta-1)(\delta-3)q^{\frac{\delta-5}{2}} + \alpha^{-\delta}\,2^{\frac{\delta}{2}-2}\,\lambda^{\frac{\delta}{2} + 2}\,q^{\frac{\delta + 3}{2}}\right){\Gamma(\frac{\delta}{2})}^{-1} (e^{\lambda t}-1)^{-\frac{\delta}{2}}\\
&\times&\exp(\displaystyle \frac{\lambda\delta t}{4} -\frac{\lambda q^2}{\alpha^2\tanh(\frac{\lambda t}{2})})\,.
\end{eqnarray*}
It is then straightforward to see that, for all $(t,q)\in \mathbf R{^*_+} \times \mathbf R{^*_+}$,
\begin{eqnarray*} 
-\theta^2\frac{\partial \eta_*}{\partial t} = -\frac{\theta^4}{2}\frac{\partial^2 \eta_*}{\partial q^2} + V\eta_* \,\,\,\, (\mathcal C_{2}^{(V)})\,.
\end{eqnarray*}
\end{proof}
\medskip
\begin{proposition}

If $\bold{x_0 > 0}$, then the density $\rho_t(q)$ of the law $Z_t$ is given by 
\begin{eqnarray*}
\rho_t(q) = \frac{4\lambda}{\alpha^2 z_0^\nu}\,\frac{e^{\lambda t(\frac{\delta}{4} + \frac{1}{2})}}{(e^{\lambda t}-1)}\,I_\nu(\frac{4\lambda z_0 e^{\frac{\lambda t}{2}}q}{\alpha^2(e^{\lambda t}-1)})\,q^{\frac{\delta}{2}}\,\exp(-\frac{2\lambda}{\alpha^2}(\frac{z_0^2 + e^{\lambda t}q^2}{(e^{\lambda t}-1)}))\displaystyle \mathbb{1}_{\{q > 0\}}\,.
\end{eqnarray*}
\end{proposition}
\begin{proof}
We have
$$
X_t = e^{-\lambda t}Y(s)\,,\quad s = \frac{\alpha^2}{4\lambda}(e^{\lambda t}-1)
$$
and
$$
\left\{
\begin{array}{rl}
Z_t&= \sqrt{X_t}\\
Z_0&= \sqrt{x_0}\,.
\end{array}
\right.
$$
For fixed $t$, for all $g$ bounded continuous, using $ (1.4)$, we get
\begin{eqnarray*}
E(g(Z_t)) &=& E\left(g(\sqrt{e^{-\lambda t}Y(s)})\right)\\
&=& \int_0^\infty g(\sqrt{e^{-\lambda t }y})\,\frac{1}{2s}\,(\frac{y}{x_0})^{\frac{\nu}{2}}\,\exp(-\frac{x_0+y}{2s})\,I_\nu(\frac{\sqrt{x_0y}}{s})dy\,.
\end{eqnarray*}
By the change of variable $z=\sqrt{e^{-\lambda t}y}$, we obtain
\begin{eqnarray*}
E(g(Z_t)) = \int_0^\infty g(z)\,\frac{1}{2s}(\frac{e^{\lambda t}z^2}{z_0^2})^{\frac{\nu}{2}}\,\exp(-\frac{z_0^2+e^{\lambda t}z^2}{2s})\,I_\nu(\frac{z_0ze^{\frac{\lambda t}{2}}}{s})\,2ze^{\lambda t}dz\,.
\end{eqnarray*}
We then find the density $\rho_t(q)$ of the law de $Z_t$ by replacing $s$ and $\nu$ by their value.
\end{proof}
\begin{proposition}
The function $\eta_*(t,q)$ defined by
\begin{eqnarray*}
\eta_*(t,q) &=& \frac{\rho_t(q)}{\eta(t,q)}\\
&=& \frac{4\lambda}{\alpha^2 z_0^\nu}\,\frac{e^{\frac{\lambda t}{2}}}{(e^{\lambda t}-1)}\,I_\nu(\frac{4\lambda z_0 e^{\frac{\lambda t}{2}}q}{\alpha^2(e^{\lambda t}-1)})\,q^{\frac{1}{2}}\,\exp(\frac{\lambda q^2}{\alpha^2} - \frac{2\lambda}{\alpha^2}(\frac{z_0^2 + e^{\lambda t}q^2}{(e^{\lambda t}-1)}))
\end{eqnarray*}
satisfies the dual equation $(\mathcal C_{2}^{(V)})$.
\end{proposition}
\begin{proof}
We have
\begin{eqnarray*}
\frac{\partial \eta_*}{\partial t} & = &
\Bigg(
(-\frac{2\lambda^2}{\alpha^2 z_0^\nu}
\, e^{\frac{\lambda t}{2}}
 \frac{(e^{\lambda t} + 1)}
      {(e^{\lambda t}-1)^2} \,
q^{\frac{1}{2}}
+\frac{8\lambda^3}{\alpha^4 z_0^\nu} 
\frac{e^{
\frac{3\lambda t}{2}}}{(e^{\lambda t}-1)^3}\,q^{\frac{5}{2}}
+ \frac{8\lambda^3}{\alpha^4 z_0^\nu}\frac{e^{
\frac{3\lambda t}{2}}}{(e^{\lambda t}-1)^3}\,z_0^2\,q^{\frac{1}{2}})
I_\nu(\frac{4\lambda z_0 e^{\frac{\lambda t}{2}}q}{\alpha^2(e^{\lambda t}-1)})\\
&-& \frac{8 z_0\lambda^3}{\alpha^4 z_0^\nu}\,e^{\lambda t}\frac{e^{\lambda t} + 1}{(e^{\lambda t} - 1)^3}
\,q^{\frac{3}{2}}\,{I_\nu}^{'}(\frac{4\lambda z_0 e^{\frac{\lambda t}{2}}q}{\alpha^2(e^{\lambda t}-1)}) 
\Bigg)
\exp\left(\frac{\lambda q^2}{\alpha^2}-\frac{2\lambda}{\alpha^2}\frac{z_0^2 + e^{\lambda t}q^2}{e^{\lambda t}-1}\right)\,,
\end{eqnarray*}
\begin{eqnarray*}
\frac{\partial \eta_*}{\partial q} &=&
\Bigg(
(\frac{1}{2} q^{-\frac{1}{2}} - \frac{2\lambda}{\alpha^2}\frac{(e^{\lambda t} + 1)}{(e^{\lambda t} -1)}\,q^{\frac{3}{2}})
\,{I_\nu}(\frac{4\lambda z_0 e^{\frac{\lambda t}{2}}q}{\alpha^2(e^{\lambda t}-1)}) + \frac{4\lambda z_0}{\alpha^2}\frac{e^{\frac{\lambda t}{2}}}{(e^{\lambda t} - 1)}\,q^{\frac{1}{2}} {I_\nu}^{'}(\frac{4\lambda z_0 e^{\frac{\lambda t}{2}}q}{\alpha^2(e^{\lambda t}-1)})
\Bigg)\\ 
&\times& \frac{4 \lambda}{\alpha^2 z_0^\nu}\frac{e^{\frac{\lambda t}{2}}}{(e^{\lambda t} - 1)}
\exp\left(\frac{\lambda q^2}{\alpha^2}-\frac{2\lambda}{\alpha^2}\frac{z_0^2 + e^{\lambda t}q^2}{e^{\lambda t}-1}\right)\,,
\end{eqnarray*}
\begin{eqnarray*}
\frac{\partial^2 \eta_*}{\partial q^2} & = &
\Bigg(
(-\frac{1}{4} q^{-\frac{3}{2}} -\frac{4\lambda}{\alpha^2}\frac{(e^{\lambda t}+1)}{(e^{\lambda t}-1)}q^{\frac{1}{2}} 
+\frac{4\lambda^2}{\alpha^4}\frac{(e^{\lambda t}+1)^2}{(e^{\lambda t}-1)^2}\,q^{\frac{5}{2}})
I_\nu(\frac{4\lambda z_0 e^{\frac{\lambda t}{2}}q}{\alpha^2(e^{\lambda t}-1)}) \\
&+&(\frac{4\lambda z_0}{\alpha^2}\frac{e^{\frac{\lambda t}{2}}}{(e^{\lambda t}-1)}\,q^{-\frac{1}{2}} -\frac{16\lambda^2 z_0}{\alpha^4}\,e^{\frac{\lambda t}{2}} \frac{(e^{\lambda t} + 1)}{(e^{\lambda t}-1)^2}\,q^{\frac{3}{2}})I_\nu^{'}(\frac{4\lambda z_0 e^{\frac{\lambda t}{2}}q}{\alpha^2(e^{\lambda t}-1)})\\
&+&\frac{16\lambda^2 z_0^2}{\alpha^4}\frac{e^{\lambda t}}{(e^{\lambda t-1})^2}q^{\frac{1}{2}}\,I_\nu{''}(\frac{4\lambda z_0 e^{\frac{\lambda t}{2}}q}{\alpha^2(e^{\lambda t}-1)})
\Bigg)
\frac{4\lambda}{\alpha^2z_0^\nu}\frac{e^{\frac{\lambda t}{2}}}{(e^{\lambda t}-1)}\exp \left (\frac{\lambda q^2}{\alpha^2}-\frac{2\lambda}{\alpha^2}
\frac{z_0^2 + e^{\lambda t}q^2}{e^{\lambda t}-1} \right)\,,
\end{eqnarray*}
and
\begin{eqnarray*}
V(t,q)\,\eta_*(t,q) &=&(\frac{A}{q^2} + Bq^2)\,\eta_*(t,q)\\
&=&(\frac{\alpha^4}{2^7}(\delta-1)(\delta-3)\,q^{-2} + \frac{\lambda^2}{2^3}q^2)\,\eta_*(t,q)\\
&=&\left(\lambda\frac{\alpha^2}{2^5}\,(\delta-1)(\delta-3)\,q^{-\frac{3}{2}} + \frac{\lambda^3}{2\alpha^2}\,q^{\frac{5}{2}}\right)\,\frac{e^{\frac{\lambda t}{2}}}{z_0^\nu (e^{\lambda t}-1)}\,I_\nu(\frac{4\lambda z_0 e^{\frac{\lambda t}{2}}q}{\alpha^2(e^{\lambda t}-1)})\\
&\times& \exp(\frac{\lambda q^2}{\alpha^2} - \frac{2\lambda}{\alpha^2}(\frac{z_0^2 + e^{\lambda t}q^2}{(e^{\lambda t}-1)}))\,.
\end{eqnarray*}
Taking into account Equation $(1.2)$ of the \textit{Bessel} function $I_\nu$, it is then straightforward to see that for all $(t,q)\in \mathbf R{^*_+}\times \mathbf R{^*_+}$,
\begin{eqnarray*} 
-\theta^2\frac{\partial \eta_*}{\partial t} = -\frac{\theta^4}{2}\frac{\partial^2 \eta_*}{\partial q^2} + V\eta_* \,\,\,\, (\mathcal C_{2}^{(V)})\,.
\end{eqnarray*}
\end{proof}

\end{document}